%% file: Paper_on_Vertex_Generated_Polytopes.tex
\pdfoutput=1
\documentclass[a4paper,11pt]{article}
\usepackage{amsfonts,latexsym,rawfonts,amsmath,amssymb,amsthm}
\usepackage{cases, cite}
\usepackage{hyperref}

\hypersetup{
    hidelinks,
}
\usepackage{graphicx}
\usepackage{setspace}
\usepackage{titlesec}
\usepackage{enumitem}
\usepackage{soul}

\usepackage{tikz}
\usepackage{scalerel}
\usepackage{pict2e}
\usepackage{tkz-euclide}
\usetikzlibrary{calc}
\usetikzlibrary{patterns,arrows.meta}
\usetikzlibrary{shadows}
\usetikzlibrary{external}
\usetikzlibrary{patterns.meta}

\usepackage{caption}
\usepackage{subcaption}

\usepackage{nicefrac}

\titleformat{\chapter}[block]
    {\normalfont\huge\bfseries}{\thechapter}{20pt}{\Huge}

\RequirePackage{color}

\theoremstyle{plain}
  \newtheorem{theorem}{Theorem}[section]
  \newtheorem{lemma}[theorem]{Lemma}
  \newtheorem{proposition}[theorem]{Proposition}
  
  \newtheorem{corollary}[theorem]{Corollary}

  \newtheorem{remark}[theorem]{Remark}
  \newtheorem{definition}[theorem]{Definition}

\theoremstyle{definition}

\numberwithin{equation}{section}

\DeclareMathOperator*{\cls}{cl}

\DeclareMathOperator*{\vol}{vol}
\DeclareMathOperator*{\intr}{int}

\DeclareMathOperator{\conv}{conv}

\DeclareMathOperator{\aff}{aff}
\DeclareMathOperator{\relint}{relint}

\newcommand{\R}{\mathbb R}
\newcommand{\VG}{vertex generated}

\newcommand{\epsl}{\varepsilon}
\newcommand{\iprod}[2]{\langle #1,#2 \rangle}

\def\sub{\subseteq}%

\def\volk#1{{\rm vol}_{#1}}

\def\RR{\mathbb{R}}%

\def\bs#1{\color{teal} #1 \color{black}}

\newtheorem{manualtheoreminner}{Theorem}
\newenvironment{manualtheorem}[1]{%
	\renewcommand\themanualtheoreminner{#1}%
	\manualtheoreminner
}{\endmanualtheoreminner}

\usepackage[nobysame,initials]{amsrefs}
\usepackage{xcolor}

\let\originalleft\left
\let\originalright\right
\renewcommand{\left}{\mathopen{}\mathclose\bgroup\originalleft}
\renewcommand{\right}{\aftergroup\egroup\originalright}

\usepackage[utf8]{inputenc}
\title{Vertex generated polytopes}

\author{Shiri Artstein-Avidan\thanks{School of Mathematics, Tel Aviv University, Tel Aviv, Israel; \href{mailto:shiri@tauex.tau.ac.il}{shiri@tauex.tau.ac.il}.}  \and Tomer Falah\thanks{Faculty of Mathematics, Technion– Israel Institute of Technology, Israel; \href{mailto:fatomer@campus.technion.ac.il}{fatomer@campus.technion.ac.il}.} \and Boaz A. Slomka\thanks{Department of Mathematics, The Open University of Israel, Ra'anana, Israel; \href{mailto:slomka@openu.ac.il}{slomka@openu.ac.il}.}}
\date{}

\begin{document}
  \maketitle

\begin{abstract}

In this paper we define and investigate a class of polytopes which we call  ``vertex generated'' consisting of polytopes which are the average of their $0$ and $n$ dimensional faces. We show many results regarding this class, among them: that the class contains all zonotopes, that it is dense in dimension $n=2$, that any polytope can be summed with a zonotope so that the sum is in this class, and that a strong form of the celebrated ``Maurey Lemma'' holds for polytopes in this class.  We introduce for every polytope a parameter which measures how far it is from being vertex-generated, and show that when this parameter is small, strong covering properties hold.  
\end{abstract}

\section{Introduction}
\renewcommand{\thefootnote}{\fnsymbol{footnote}} 
\footnotetext{\emph{Key words and phrases.} Brunn-Minkowski inequality, convex bodies, polytopes, zonotopes, covering.}
\renewcommand{\thefootnote}{\arabic{footnote}} 
\renewcommand{\thefootnote}{\fnsymbol{footnote}} 
\footnotetext{\emph{2010 Mathematics Subject Classification.} 52A40, 52B12, 52C17.}
\renewcommand{\thefootnote}{\arabic{footnote}}

Given a polytope $P\subset \RR^n$ we let $V(P)$ denote the set of its vertices. Our starting point is the following simple observation which  follows directly from  Caratheodory's theorem: for
every polytope $P\sub \R^n$,  
 	\begin{equation}\label{eq:carat} P=\frac{n}{n+1} P+\frac 1{n+1}V(P).\end{equation}
 \noindent Here $A+B = \{a+b: a\in A, b\in B\}$ denotes the Minkowski sum of the sets $A$ and $B$. 
Indeed, every $x\in P$ can be written as $x = \sum_{i=1}^{n+1}\lambda_i v_i$
for some $v_i \in V(P)$ and $\lambda_i \ge 0$ with $ \sum_{i=1}^{n+1}\lambda_i = 1$. Assuming without loss of generality that $\lambda_i$ are decreasing, so that $\lambda_1 \ge  {1}/({n+1})$, we get $x = v_1/({n+1}) +  {ny}/({n+1})$ where $y=(({n+1})/{n})(\sum_{i=2}^{n+1}\lambda_i x_i + (\lambda_1-1/({n+1}))v_1)\in P$. 


In \cite{Schneider-1678} Schneider defined the following parameter, measuring in an affine way the deviation of a compact set $A$ from its convex hull: 
\[ c(A) = \inf\{ t \ge 0: A + t\, {\rm conv}(A) {\rm~is~convex}\}
= 
\inf\{ t \ge 0: A + t \,{\rm conv}(A)  
= (1+t) {\rm conv}(A)\}. \]

In this note the set $A$ will be the set of vertices of a given polytope, and it will be more convenient to use a slightly different normalization. We define 
 \begin{definition}\label{def:lambdaP}
	For a polytope $P\sub \RR^n$ define 
	\[ \lambda(P) = \sup \{ \lambda : P=(1-\lambda) P+\lambda V(P)\}.\]
\end{definition}
\noindent Thus, 
$\lambda(P) = 1/ ({1+c(V(P))})$. The above observation \eqref{eq:carat} implies that for any polytope $P\sub \R^n$ one has $\lambda(P) \ge {1}/({n+1})$ (equiv. $c(A)\le n$). 
The only polytopes  for which $\lambda(P) = 1/(n+1)$ are simplices. This was proved by Schneider in \cite{Schneider-1678} see \cite{Sch_book}*{Theorem 3.1.9}, see Proposition \ref{prop:simplex iff lambda is minimal possible} for an alternative proof.

We call a polytope $P\sub\RR^n$   ``$\lambda$-\VG'' if $\lambda(P) \ge \lambda$, 
and we denote the class of $\lambda$-vertex generated polytopes by ${\rm VG}(\RR^n, \lambda)$.  Clearly ${\rm VG}(\RR^n, \lambda)$  is a decreasing family in $\lambda$.  It is easy to check that $\lambda(P) \le 1/2$ for any polytope $P$ by considering a $1$-dimensional face of $P$. 
We  let ${\rm VG}(\RR^n) = {\rm VG}(\RR^n, 1/2)$ be the smallest of these classes, and call the polytopes in this class ``\VG''. 
Our main objective in this note is to study the class 
  ${\rm VG}(\RR^n)$, namely polytopes for which 
  \[ P + V(P) = 2P. \]

The name ``$\lambda$-\VG'' stems from the fact that the equality in the   definition of $\lambda(P)$ can be iterated. As we show in Proposition \ref{prop:infinite-sum}, $P\in {\rm VG}(\RR^n,\lambda)$ if and only if
 \[ 
 P = \cls\left({\sum_{i=0}^\infty (1-\lambda)^{i} \lambda V(P)}\right),
 \]   
where $\cls(A)$ denotes the closure of the set $A$. 

This paper is devoted to studying some interesting features of these  classes. We present several of these in the introduction, and the text contains several other results and observations. 
Our first main result is that for any $\lambda$, the sum of a polytope in ${\rm VG}(\RR^n, \lambda)$ with a segment is also in ${\rm VG}(\RR^n, \lambda)$.

 \begin{theorem}\label{thm:adding-a-seg} Let $n\ge 1$ and $\lambda \in [1/(n+1), 1/2]$. Given $x,y\in \RR^n$ denote $\ell = [x,y] = \{ (1-\mu)x + \mu y: 0\le \mu \le 1\}\subset \RR^n$. For any $P\in {\rm VG}(\RR^n, \lambda)$ it holds that $P+\ell\in {\rm VG}(\RR^n, \lambda)$ as well.  
 \end{theorem}

As a consequence, we get that zonotopes are \VG. Recall that a zonotpe in $\RR^n$ is the Minkowski sum of a finite number of segments. The class of zonotopes is well studied, see \cite{Sch_book} and references therein. Indeed, Theorem \ref{thm:adding-a-seg} implies that all zonotopes are \VG, since a segment is clearly \VG.  
 \begin{corollary}\label{cor:zonotope}
      Let $n\ge 1$. 
 	Let $Z\subset \RR^n$ be a zonotope. Then $Z\in{\rm VG}(\RR^n)$.
 \end{corollary}

In the spirit of ``generalized zonoids'', see \cite{Sch_book}, we show that all polytopes are so-called ``generalized ${\rm VG}(\RR^n)$'', namely that any polytope can be summed with a zonotope and become \VG. 

\begin{theorem}\label{thm:zon}
Let $n\ge 1$, and let $P\subseteq \RR^n$ be a polytope. Then there exists a zonotope $Z\sub\RR^n$ such that $P+Z$ is \VG.
\end{theorem}

A natural question is whether any polytope can be approximated by \VG~polytopes. We verify this in the plane, and prove that in dimension $n=2$ \VG~polytopes are dense  with respect to the Hausdorff metric, in the class of all convex bodies. 
 \begin{theorem}\label{thm:denseinR2} Let $P\sub\RR^2$ be a compact convex set. Then there exists a sequence $(P_m)_{m=1}^{\infty} \subset {\rm VG}(\RR^2)$ with $P_m \to P$ in the Hausdorff metric. 
\end{theorem}

Theorem \ref{thm:adding-a-seg} implies that the Minkowski sum of a \VG~polytope and a zonotope remains \VG. 
It is not clear whether $ {\rm VG}(\RR^n, \lambda) $ is generally closed under Minkowski addition (nor whether it is closed under projections). However, we are able to show that this holds for any pair of centrally symmetric polytopes which are in a  ``generic''  position. To  
formulate the result, let $N_P(v)$ denote the cone of normal vectors of a polytope $P$ at a vertex $v$, namely 
 	\[ N_P(v)= \{y\in \RR^n: \langle{y},{x-v}\rangle\le 0,\,\,\, \forall x\in P\},\]
where $\iprod{\cdot}{\cdot}$ is the standard scalar product on $\RR^n$. 
A  ``generic pair" is now defined as follows. 
\begin{definition}\label{def:gen-pair}
 	Given two polytopes $P,Q\sub \RR^n$ with non-empty interior, we  say that they are a ``generic pair'' if, 
  given  $v\in V(P)$ and $w\in V(Q)$, it holds that  
  $$
N_P(v)\cap N_Q(w) \neq \emptyset\implies {\rm int}(N_P(v))\cap {\rm int}(N_Q(w)) \neq \emptyset.
  $$
 \end{definition}

 \begin{proposition}\label{prop:sum_of_generic_VG}  Let $n\ge 1$ and $\lambda \in [1/(n+1), 1/2]$. Suppose that $P,Q\in {\rm VG}(\RR^n, \lambda)$ are centrally symmetric that form a generic pair. Then    $P+Q\in {\rm VG}(\RR^n, \lambda)$ as well. 
	\end{proposition}

The paper is organized as follows. In Section \ref{sec:obst} we show that a $\lambda$-\VG~polytope must have exponentially  many vertices and that faces of $\lambda$-\VG~are  also $\lambda$-\VG.  In Section \ref{sec:zon}, we prove Theorem \ref{thm:adding-a-seg} in a slightly more general form.  In Section  \ref{sec:dense} we discuss various denseness notions, we prove Theorem \ref{thm:denseinR2},  
we show that, with respect to a weaker metric, $\lambda$-\VG~polytopes are closed and we prove Theorem \ref{thm:zon}. Section \ref{sec:central_sym} is devoted to centrally symmetric polytopes. Along with several other observations, we show that a centrally symmetric polytope is $\lambda$-\VG~ if and only if all of its faces are, and we prove Proposition \ref{prop:sum_of_generic_VG}. We include a curious fact regarding other linear variants of \VG~polytopes, demonstrating that $P-V(P) = P-P$ can occur only when $P$ is centrally symmetric and \VG. Section \ref{series} is devoted to the series expansion and covering properties of $\lambda$-\VG~polytopes. Section \ref{sec:other_stuff} includes some additional remarks and connected results, as well as a proof of the fact that the simplex is the unique minimizer of $\lambda(P)$. 

\subsubsection*{Acknowledgement}
 {The authors were partially supported by ISF grant number 784/20. The first and second named authors were also partially supported by European Research Council (ERC) under the European Union's Horizon 2020 research and innovation programme (grant agreement No. 770127).}

\section{Some simple obstacles}\label{sec:obst}

We will demonstrate in this paper that the class ${\rm VG}(\RR^n, \lambda)$, for any $\lambda$, is quite rich. In the next section we will see that it
includes all zonotopes, which follows from the stronger claim which is that if  a polytope is $\lambda$-\VG, then it remains $\lambda$-\VG~also when  adding a segment  to it. In other words, the class of $\lambda$-\VG~polytopes is closed under the operation of Minkowski summation with a segment (and, by induction, under adding any zonotope). 

It is clear, however, 
that not every polytope is \VG (for example, a triangle). In fact, in any dimension $n$, the simplex is not a member of ${\rm VG}(\RR^n, \lambda)$ for any $\lambda >1/(n+1)$, as we show in Proposition \ref{prop:simplex iff lambda is minimal possible}. We also show there that any polytope which is not a simplex, does belong to some ${\rm VG}(\RR^n, \lambda)$  with $\lambda >1/(n+1)$.
%
We begin with some simple obstacles to belonging to the class $  {\rm VG}(\RR^n, \lambda)$.

\begin{lemma}\label{lem:vert}
Let $n\ge 1$ and $\lambda \in [1/(n+1), 1/2]$. If $P\in {\rm VG}(\RR^n, \lambda)$ has non-empty interior then it has at least $1/(1-\lambda)^n$ vertices. Moreover, if $P\in {\rm VG}(\RR^n)$ has exactly $2^n$ vertices, then it must be a linear image of the cube.
\end{lemma}

\begin{proof}
Indeed, by assumption $P\sub \bigcup_{v\in V(P)} (\lambda v+(1-\lambda)P)$ and so comparing volumes  $$\vol(P) \le |V(P)|\vol((1-\lambda)P) = |V(P)|(1-\lambda)^n \vol(P).$$
This completes the proof of the inequality. 
Next, assume that $P\in {\rm VG}(\RR^n)$ has exactly $2^n$ vertices. This means that each of the intersections $(\frac{v+P}{2})\cap(\frac{w+P}{2})$ is non-empty (includes $(v+w)/2$) and has measure zero. Therefore, considering a separating hyperplane, there exists some unit vector $u\in \RR^n$ such that $u$ is an outer normal to $P$ at $v$ and $-u$ is an outer normal of $P$ at $w$. A set on $\partial P$ with this property (that for any two of its elements one can find such a $u$) is called an ``antipodal set''. In other words, the vertices of $P$ form an antipodal set of cardinality $2^n$. By a result of Danzer and Gr\"unbaum \cite{DG62}, this happens only if $P$ is a linear image of the cube. This completes the proof. 
\end{proof}
\noindent We remark that by the same argument,   $P\in {\rm VG}(\RR^n, \lambda)$ with non-empty interior   has strictly more than $1/(1-\lambda)^n$ vertices. 

Next, we claim that faces of $\lambda$-\VG~polytopes are also $\lambda$-\VG. 
In particular, any polytope with a triangular $2$-dimensional face is not \VG. 
We prove a formally stronger claim, namely that if a face of $P$ is included in $(1-\lambda) P+ \lambda V(P)$, then the face is $\lambda$-\VG 
 (even without requiring the whole polytope to be in ${\rm VG}(\RR^n, \lambda)$). 
\begin{lemma}\label{lem:fac}
  Let $n\ge 1$ and $\lambda \in [1/(n+1), 1/2]$.  Let $P\sub \RR^n$ be a  polytope and let $F$ be a face of $P$. Assume $F \subseteq (1-\lambda) P+ \lambda V(P)$. Then $F = (1-\lambda) F+ \lambda V(F)$.   In particular, faces of $\lambda$-\VG~polytopes are $\lambda$-\VG.
\end{lemma}

\begin{proof}[Proof of Lemma \ref{lem:fac}]
    Let $x \in F$. By our assumption  there exist $v \in V(P),\, y \in P$ such that $x = (1-\lambda) y+ \lambda v$. Clearly both $y$ and $v$ belong to $F$. 
    A vertex of $P$ which belongs to $F$ is also a vertex of $F$, and the proof is complete.  
\end{proof}

\begin{remark} 
We  note that ${\rm VG}(\R^n, \lambda)$ is not closed under intersection with either other $\lambda$-\VG~polytopes or a subspace, because we can translate and rotate two cubes so that their intersection is a simplex, and as all centrally symmetric polytopes are  sections of higher-dimensional hyper-cubes. 
\end{remark}

\section{Zonotopes, and Minkowski addition of a segment} \label{sec:zon}

In this section we will prove that the sum of a $\lambda$-\VG~polytope and a segment is $\lambda$-\VG. In fact, we show something slightly more general, namely that for a polytope $P$ all of whose faces are $\lambda$-\VG,
the difference between the polytope and the average  $(1-\lambda)P+\lambda V(P)$ only decreases when passing to $P+\ell$ for any line segment $\ell$. More precisely we prove the following.  

\begin{theorem}\label{thm:sum}
Let $n\ge 1$ and $\lambda \in [1/(n+1), 1/2]$.
Let $P\sub \R^n$ be a polytope such that all of its $(n-1)$-dimensional faces are $\lambda$-\VG. Let $\ell\sub \RR^n$ be an origin symmetric line segment and let 
$P' = P + \ell$.  Then 
\begin{equation}
	P'\setminus\left( (1-\lambda)P'+\lambda V(P')\right) \subseteq P\setminus\left( (1-\lambda)P+\lambda V(P)\right). 
\end{equation}
 In particular, if  $P$ is $\lambda$-\VG~(so that the right hand side is empty) then so is $P+\ell$ for any segment $\ell$.
\end{theorem}

To prove Theorem \ref{thm:sum} we  need the following simple lemma.
\begin{lemma}\label{lem:sec}
  Let $n\ge 1$, $P\sub \RR^n$   a polytope,  $F$   a face of $P$ and  $\theta \in S^{n-1}$ satisfy that   $P \cap (\epsl  \theta + F) = \emptyset$ for any $\epsl>0$. Then for any $c>0$,  $F+c\theta$ is a face of $P+[-c\theta,c\theta]$.
\end{lemma}

\begin{proof}
	We will use the following notation: for a polytope $P$ and a unit vector $u$, we let $F_u$ denote the face of $P$ in direction $u$, namely 
	\[ F_u(P) = \{ x\in P: \langle u, x \rangle = \sup_{y\in P}\langle u, y \rangle  =  h_P(u)\}.\]
It is well known (see \cite{Sch_book}*{Theorem 1.7.5}) that Minkowski addition respects this definition, namely that for $u \in S^{n-1}$ and polytopes $P,Q \subseteq \R^n$ one has 
\begin{equation}\label{eq:Fu_sum}
	F_u(P+Q)=F_u(P)+F_u(Q).
\end{equation}
The condition that  $P \cap (\epsl \theta + F) = \emptyset$ for any $\epsl>0$ is equivalent to the existence of some $u \in S^{n-1}$  which belongs to the normal cone of  $F$ (that is, $F=F_u(P)$) and such that $\langle u, \theta \rangle > 0$.
 Thus
 \[F_u(P+[-c\theta,c\theta])=F_u(P)+F_u([-c\theta,c\theta])=F+c\theta.\]
 This completes the proof of the lemma. 
\end{proof}

\begin{proof}[Proof of Theorem \ref{thm:sum}]
Given $\ell$ denote  $\ell=[-c\theta,c\theta]$ where $\theta \in S^{n-1}$ and $c>0$. We start by showing that 
$(1-\lambda)P+\lambda V(P) \subseteq (1-\lambda)P'+\lambda V(P') $. Let $x \in (1-\lambda)P+\lambda V(P)$, and choose $v \in V(P)$, $y\in P$ such that $(1-\lambda)y+\lambda v= x$. 
Since $\{v\}$ is a $0$-dimensional face of $P$, either $(v+\RR^+ \theta) \cap P = \emptyset$ or $(v-\RR^+\theta) \cap P = \emptyset$ (or both). By interchanging $\theta$ and $-\theta$ we may assume without loss of generality the former, and by 
  Lemma \ref{lem:sec} we get that $v+c\theta \in V(P+\ell)= V(P')$. Using that $y-\frac{\lambda }{1-\lambda}c\theta \in P+\ell = P'$, we get 
    \[x = (1-\lambda)y+\lambda v = (1-\lambda) (y-\frac{\lambda }{1-\lambda} c\theta)+\lambda (v+c\theta) \in   (1-\lambda)P'+\lambda V(P') .\]
Next, let $x\in P'\setminus  ((1-\lambda)P'+\lambda V(P') )$. We have just demonstrated that $x\not\in (1-\lambda)P +\lambda V(P )$, and are left with showing that $x\in P$. Assume towards a contradiction that $x\not\in P$. Since $x\in P+\ell$ it follows that  $P \cap (x+[-c\theta,c\theta]) \ne \emptyset$, and (again, switching to $-\theta$ if needed) we may assume  $P \cap (x+[-c\theta,0)) \ne \emptyset$. We let 
  $c':=\inf\{r>0:x-r\theta \in P\}$ and $x'=x-c'\theta$. In other words, $x'$ is one of the two intersection points of $x+\RR\theta$ and $\partial {P}$, the one closer to $x$. 
  Denote by $F$ the minimal face of $P$ containing $x'$ so that $x'\in \relint(F)$. By our assumption, $F = (1-\lambda)F + \lambda V(F)$  so that there exist $v\in V(F),\, y\in F$ such that $ (1-\lambda)y + \lambda v =x'$. 
    Since $x'\in \relint(F)$ and $c'$ was minimal, the direction $\theta$ must satisfy that  there is some $u\in S^{n-1}$ in the normal cone of $F$ with $\iprod{u}{\theta}>0$ and thus  $(F+\epsl\theta )\cap P = \emptyset$ for any $\epsl>0$.  We  use  Lemma \ref{lem:sec} to conclude that $v+c\theta \in V(P+\ell)$, and since $0< c' \le c$ we have that 
    $ \frac{c'-\lambda c}{1-\lambda }\in [-c,c]$ and so
    $y+\frac{c'-\lambda c}{1-\lambda }
  \theta\in P+\ell$. Therefore
    \[x = x'+c'\theta= 
    (1-\lambda)( y + \frac{c'-\lambda c}{1-\lambda }
  \theta)  + \lambda (v + c\theta) \in  (1-\lambda)P'+ \lambda V(P').\]
This contradicts our assumption on $x$, and the proof is complete.
\end{proof}

Using that a segment is \VG, and Theorem \ref{thm:sum} which implies that when we add a segment to a \VG~polytope it remains \VG, we get that any zonotope in any dimension is \VG, proving Corollary \ref{cor:zonotope}. 
\noindent We will see another simple proof that zonotopes are \VG~in  Section \ref{sec:anpotherproofforzn}.

\section{On various denseness notions for  \texorpdfstring{${\rm VG}(\RR^n)$}{VG(Rn)} and  \texorpdfstring{${\rm VG}(\RR^n,\lambda)$}{VG(Rn)lambda}} \label{sec:dense}

It is natural to ask whether \VG~polytopes are dense within the class of polytopes, or, similarly, within the class of all convex bodies (when considering the Hausdorff metric, this is the same question). We are able to prove this in dimension $n=2$, which is the subject of Section \ref{sec:dim2dense}. 
In Section \ref{sec:newmetric} we define a different metric on polytopes, $d_F$, which is weaker than the Hausdorff metric $d_F(P, Q) \ge d_H(P,Q)$, 
and such that with respect to this metric the \VG~polytopes form a closed set. 
The metric is quite natural and is given by the Hausdorff distance between the vertex-sets of the polytopes. 
We consider other similar questions, such as whether one can always add a zonotope to a polytope so that the resulting polytope is \VG~(we show this is correct, see Theorem \ref{thm:zon}). In the same vein, one can try to add a very {\em small} polytope such that the resulting  polytope is \VG, which, if true, would imply {\bf{\bs{density}}} with respect to Hausdorff distance. However, this we show cannot hold, and if a polytope $P\sub \RR^n$ is not \VG~then there is some $\epsl>0$ such that for all polytopes  $Q\sub \epsl B_2^n$ the polytope $P+Q$ is not \VG. These two results are given in Section \ref{sec:minkadd}.

\subsection{Vertex generated polytopes in the plane}\label{sec:dim2dense}

In this section, we focus on  planar vertex generated polytopes, showing that they are dense in the class of all planar convex bodies:

\begin{theorem}\label{thm:dens}
	The class ${\rm VG}(\RR^2)$ of  \VG~polytopes in the plane  is  dense, with respect to the Hausdorff distance, in the class of all planar convex bodies.
\end{theorem}

We shall make use of the following lemma (which actually  holds in any dimension). 
\begin{lemma}\label{lem:triv}
	Let $P\sub \R^2$ be a polytope  with non-empty interior. For any   $u \in \partial P$, there exists $r>0$ such that 
	\[(u+rB_2^2)\cap P = (u+rB_2^2)\cap \frac 12(P+u).\]
\end{lemma}
\begin{proof}
	We denote the $1$-dimensional faces of $P$ by $E(P)$, that is, edges of $P$. 
	Fix some $u\in \partial P$ and denote
	\[r= \frac 12\min \{d(u,E): E\in E(P), u\not\in E \}>0.\]
	Given $x \in (u+rB_2^2)\cap P$,  denote by $u'\in\partial P$ the  point farthest from $u$ for which $x\in [u,u']$. It is important to note that $u'\neq x$ as by the definition of $r$, in the only relevant case which is $x\in\partial P$, clearly $x$ and $u$ must be on the same edge, but  $x\not\in V(P)$ and there will be a vertex $u'$  of $P$ farther from $u$ with $x\in [u,u']$. 
	By our choice of $r$, we have  $d(u,u')\ge 2r$ (again, one can distinguish the two cases, $x\in \partial P$ and $x\in {\rm int}(P)$). Since $d(x,u)\le r$ we see that in fact $x\in [u,  (u+u')/2]$. However, $[u, (u+u')/2]\sub (P+u)/2 $ so that $x \in (P+u)/2$, as needed. 
\end{proof}

\begin{proof}[Proof of Theorem \ref{thm:dens}]
	Let $P$ be a polytope and $\epsl > 0$. 
	Clearly, we may also assume that $P$ is not \VG, so that, in particular, $P$ is not centrally-symmetric up to any translation (recall that in that case, $P$ would be a zonotope, which is \VG, by Corollary \ref{cor:zonotope}). 
	Our goal is to construct a polytope $Q \in {\rm VG}(\RR^2)$ such that $d_H(P,Q)\le \epsl$.
	
	By Lemma \ref{lem:triv}, there exists some $\delta>0$ such that for every vertex $v$ of $P$ it holds that   $P\cap\intr(v+\delta B_2^2)\sub(P+V(P))/2$. 
	
	Denote the set of vertices of $P$ by $\{v_i\}_{i=1}^m$, and set $v_{m+1}=v_1$. Assume that the vertices are labeled so  that $E_i=(v_i,v_{i+1})$ is an edge of $P$ (i.e., a $1$-dimenional face) for each $i\in\{1,\dots,m\}$. Let $\eta_i$  be the outer normal of $E_i$ and $\theta_i$ be an outer normal of $v_i$ which is different  from all the outer normals $\{\eta_j\}_{j=1}^m$. For each $i\in\{1,\dots,m\}$, consider a circle of sufficiently large radius such that  $E_i$ is a chord with corresponding (minor) arc $A_i$ and the following conditions hold as well (see Figure \ref{fig:dense2D}):
	\begin{enumerate}[label=(\text{C\arabic*})]
		\item\label{e1}$d_H(A_i,E_i)<\epsl$,
		\item\label{e2}$2e_i-A_i\sub P$, where $e_i=\frac {v_i+v_{i+1}}2$ is the center of the edge $E_i$, 
		\item\label{e3} {$A_i\cap \{v_i+\theta_i^\perp\}=v_i$ and $A_i\cap  \{v_{i+1}+\theta_{i+1}^\perp\}=v_{i+1}$}.
	\end{enumerate}
	By   \ref{e3}, we see that  $Q'=\conv\left(\bigcup_{i=1}^m A_i\right)\supseteq P$ is a compact convex set, and that every  boundary point of $Q'$ is an  extremal point of $Q'$. Moreover, since $P$ is not centrally-symmetric up to any translation, we know that $Q'$ is not centrally-symmetric up to any translation, as well.
 
	Consider the family $\{S_u\}_{u\in\partial Q'}=\{\intr(u+Q')/2\}_{u\in\partial Q'}$. We claim that  it forms a cover of $\intr(Q')$ and therefore a cover of  the  compact set 
 $$P':=P\setminus(\bigcup_{v\in V(P)}\intr(v+\delta B_2^2))\sub \intr(Q'),$$
where $\delta>0$ is the positive constant chosen at the beginning of the proof.  
 
	To show that this is indeed a cover, let $x\in\intr(Q')$ and suppose that $x\in (a,b)$ for some $a,b\in\partial Q'$. Note that if $|x-a|<|x-b|$ then $x\in(a,(a+b)/2)$ and hence  $x\in B_a$. This means that $x$ is either in some $S_u$ or is the center of every line section of $P$ passing through $x$. The latter would imply that $Q'-x$ is centrally-symmetric, and as  we assumed that this is not the case,  we  conclude that $\{S_u\}_{u\in\partial Q'}$ is indeed a cover of $Q'$.

 \input{dense2D}
 
	By the compactness of $P'$, we  may select a finite sub-cover $\{S_{u'_j}\}_{j=1}^n$ of   $P'$ such that $V(P)\sub \{u'_j\}_{j=1}^n$. Let $Q_k$ be the convex hull of $k\ge n$ points $\{u'_j\}_{j=1}^n\cup\{u_j\}_{j=1}^{k-n}$ on the boundary of $Q'$ such that $d_H(Q_k,Q')\to 0$.  As one can verify, using a standard compactness argument, for some $k_0$ big enough, the family $\{\intr(u'_j+Q_{k_0})/2\}_{j=1}^n$ forms a cover of $P'$ as well.  
	%
	We define our desired polytope by $Q=Q_{k_0}$.  Clearly, by \ref{e1}, we have $d_H(P,Q)<\epsl$. Also note that, by our construction, $V(P)\sub \{u'_j\}_{j=1}^n\sub V(Q)$.
	
	It is left to show that $Q$ is vertex generated.  Let $x\in Q$. Suppose first  that $x\in P$. If $x\in \intr(v+\delta B_2^2)$ for some $v\in V(P)$ then, by our choice of $\delta>0$,  $x\in (P+V(P))/2$, and since $V(P)\sub V(Q)$, it follows that $x\in (Q+V(Q))/2$, as well. Otherwise,  $x\in P'$ and then, by our choice of $\{u'_j\}$ in the  construction, $x\in (u'_j+Q)/2$ for some vertex $u'_j\in V(Q)$.
	
	Suppose next that $x\in Q\setminus P$. In this case,   $x\in \conv(A_i\cup E_i)$ for some $i$. Let $y\in\partial Q$ be the intersection of the ray emanating from $e_i$ and passing through $x$ with $\partial Q$. 
	Clearly, we also have $y\in \conv\{A_i\cup E_i\}$. If $y \in V(Q)$ then by \ref{e2},  $2x-y\in Q$, and hence $x \in \frac 12(V(Q)+Q)$. 
	Otherwise, let $a,b \in V(Q)$  be the vertices of the edge containing $y$.  By definition, $a,b \in A_i$ and so, by \ref{e2},  we have $2e_i-a,2e_i-b\in Q$. By interchanging the roles of $a$ and $b$, we may assume that  $x \in \conv(e_i,\frac12(a+b),a)$ and so  $2x-a \in\conv\{a,b,2e_i-a\}\subseteq Q$. Consequently, we get that  $x \in \frac 12(Q+V(Q))$, which  completes our proof.
\end{proof}

\color{black}

\subsection{Non-denseness with respect to a non-standard metric}\label{sec:newmetric}

While the question of denseness of \VG~polytopes with respect to the Hausdorff metric remains open in dimensions $n\ge 3$, we may define a weaker metric with respect to which  \VG~polytopes form a closed set within the class of all polytopes. This metric, which we denote $d_F$, is defined for a pair of polytopes as the Hasudorff distance between the sets of their vertices. Since one may approximate a segment, say, in the Hausdorff metric, by a sequence of triangles which are the convex hulls of the segment with a vertex close to the middle of the segment, we see that $d_F$ is weaker that $d_H$ (as this is an example where $d_H\to 0$ and $d_F\not\to 0$). Let us formally introduce this new metric on polytopes.

\begin{definition}
Let $n\ge 1$. 	Given polytopes $P,Q\subseteq \RR^n$ with vertex sets $V(P)$ and $V(Q)$ respectively, let $d_F(P,Q) = d_H(V(P), V(Q))$ or, equivalently, 
	\[ d_F(P,Q) = \min\{\epsl>0: V(Q) \sub  V(P)+ \epsl B_2^n\quad{\rm and}\quad V(P)  \sub  V(Q)+ \epsl B_2^n\}.\] Clearly $d_F(P, Q) \ge d_H(P, Q)$ (since $V(P)+\varepsilon B_2^n \subseteq P+ \epsl B_2^n$). 
\end{definition}

It turns out that with respect to this metric and for any $\lambda$, the class  ${\rm VG}(\R^n,\lambda)$ is closed. 

\begin{proposition}
Let $n\ge 1$ and $\lambda \in [1/(n+1), 1/2]$.	For any polytope $P\sub \RR^n$ which is not $\lambda$-\VG, there exists some $\epsl>0$ such that if $d_F(P, Q)<\epsl$ then $Q$ is not $\lambda$-\VG~as well. Equivalently, for polytopes $P,(P_m)_{m=1}^\infty$ in $\RR^n$,  if $P_m\in {\rm VG}(\R^n,\lambda)$ and $d_F(P_m , P)\to 0$ then $P\in  {\rm VG}(\R^n,\lambda)$.   
\end{proposition}
\begin{proof}
	Consider a sequence of polytopes $(P_m)_{m=1}^\infty\subseteq {\rm VG}(\R^n,\lambda)$ and a  polytope $P\sub \RR^n$ such that $d_F(P_m,P)\longrightarrow 0$. In particular $d_H(P_m,P)\to 0$ and $P_m$ are uniformly bounded. Given $x \in P$,   one may find a sequence $x_m \in P_m$ such that $x_m\to x$. As $P_m\in {\rm VG}(\R^n,\lambda)$ we can write  $x_m =  (1-\lambda)y_m+\lambda v_m$ with 	
	$y_m\in P_m$ and $v_m \in V(P_m)$. Since $v_m$ and $y_m$ are bounded,  there exists sub-sequences $v_{m_i}$ $y_{m_i}$ which converge. We  call their limits $v,y$ respectively, and note that $x = (1-\lambda)y+\lambda v$. Using again that $P_m\to P$ in the Hausdorff distance and $V(P_m)\to V(P)$ in the Hausdorff's distance, it follows that $v \in V(P)$ and $y \in P$. We thus see $P = (1-\lambda)P + \lambda V(P)$, as claimed.  
\end{proof}

\subsection{Generalized \VG~polytopes}\label{sec:minkadd}\label{sec:gen_zon}

We start with the result that any polytope is a so-called ``generalized \VG~polytope'' namely that  for any polytope one can find some zonotope such that their Minkowski sum is \VG. The name is inspired by the notion of a ``generalized zonoid'' from convexity (see \cite{Sch_book}*{Section 3.5}) which is defined to be a convex body $K$ such that there exist two zonoids $Z_1, Z_2$ with $K+Z_1 = Z_2$ (a zonoid is defined to be a limit of zonotopes in the Hausdorff metric, for definition and many interesting facts see \cite{Sch_book}*{Section 3.5}). For Zonoids, however, if a polytope is a generalized zonoid then it must be a zonotope to begin with (see Corollary 3.5.7 in \cite{Sch_book}) whereas for our question, any polytope can be summed with a \VG~polytope to obtain a \VG~polytope.   On the other hand, Theorem \ref{thm:zon} is similar in spirit to the fact that generalized zonoids are dense within the class of centrally symmetric convex bodies in the Hausdorff metric (see \cite{Sch_book}*{Corollary 3.5.7}).

\begin{manualtheorem}{\ref{thm:zon}}
Let $n\ge 1$. For any polytope $P\subseteq \RR^n$ there exists a zonotope $Z\sub\RR^n$ such that $P+Z$ is \VG.
\end{manualtheorem}
\begin{proof}
	We  prove the claim by induction on the dimension $n$. For $n=1$ there is nothing to prove as a segment is \VG. 
	
	Next, assume   the statement of the theorem holds for any dimension $k<n$, and let $P\sub \RR^n$ be a polytope. Its boundary is the union of a finite number of faces $F_i$ of dimension $k<n$, and so by the inductive assumption we may, for each $i$,  find a zonotope $Z_i$ such that $F_i+Z_i $ is \VG.  
	Denote $Z = \sum Z_i$, and by \eqref{eq:Fu_sum} we have for any $\theta$
	\[F_{\theta}(P+Z) = F_{\theta}(P)+\sum F_{\theta}(Z_i).\]
We do not claim that $P' = P+Z$ is \VG, but only that it satisfies the condition $\partial P' \sub (P'+ V(P'))/2$ (note that for $n=2$ this condition is satisfied automatically as all facets are segments). 
	To show this containment, consider some vector $\theta\in S^{n-1}$, and let $i=i(\theta)$ denote the index of the corresponding face of $P$, namely $F_\theta(P) = F_i$. 
	Because $Z_i\sub \theta^\perp$, clearly  $F_{\theta}(Z_{i})=Z_{i}$, which means in particular that 
	\[F_{\theta}(P+Z) = F_{i}+Z_{i}+\sum_{j\ne i}F_{\theta}(Z_j).\]
	But $F_{\theta}(Z_j)$ are zonotopes and $F_{i}+Z_{i}\in {\rm VG}(\RR^n)$, and so by Theorem \ref{thm:sum} also their sum is \VG. This shows that the face of $P'$ in direction $\theta$ is \VG. Since $\theta$ was arbitrary, this shows that every face of $P'$ is \VG, and we get that $\partial P' \sub (P'+ V(P'))/2$. 
	
	We  next find a zonotope $Z'$ such that $P'+ Z'$ is \VG. 
	Choose $\eta \in S^{n-1}$ such that 
	\begin{equation}\label{highdvertex}
		F_\eta(P')=\{u_1\},\,F_{-\eta}(P')=\{u_2\}
	\end{equation} for some $u_1,u_2 \in V(P')$.
(namely $\eta, -\eta$ are each in the interior of the normal cone of some vertex, this is true for almost every $\eta$).

	Note that $P' \sub [u_1,u_2]+\eta^\perp$, so there exists a zonotope $Z' \sub \eta ^\perp$   large enough such that $P' \sub [u_1,u_2]+Z'$. 
	Since a zonotope is a sum of line segments, we see that, by inductively using Theorem  \ref{thm:sum} (the conditions of which are satisfied by our construction of $P'$) we have
	\begin{equation}\label{eq:pplusz}(P'+Z')\setminus\frac 12(P'+Z'+V(P'+Z'))\subseteq P'\setminus \frac12 (P'+ V(P'))\sub P'.\end{equation}
Note that by construction $F_\eta(Z')=F_{-\eta}(Z')=Z'$ and hence, by \eqref{highdvertex} and \eqref{eq:Fu_sum}, we have $F_\eta(P'+Z')=u_1+Z'$ and $F_{-\eta}(P'+Z')=u_2+Z'$. Since  vertices of a polytope's face are vertices of the polytope itself we get
	\[V([u_1,u_2]+Z')\subseteq V([u_1,u_2])+V(Z')= V(u_1+Z')\cup V(u_2+Z')\subseteq V(P'+Z').\]
	Finally from the fact $[u_1,u_2]+Z'$ is a zonotope,  Corollary \ref{cor:zonotope} implies that it is \VG, 
	and therefore
	\[ P' \sub [u_1,u_2]+Z' \subseteq \frac12 ( [u_1,u_2]+Z' + V([u_1,u_2]+Z' ))\subseteq 
	\frac12 ( P'+Z' + V(P'+Z' )).	
	\]
 Joining this with \eqref{eq:pplusz}, we see that $(P'+Z')\setminus (P'+Z'+V(P'+Z'))/2$ must be empty, meaning that $P'+Z'$ is \VG, and the
	proof is complete. 
\end{proof} 

\begin{remark}\label{rem:dim}
Let us remark on the degrees of freedom in our final choice of $Z$ in the proof of Theorem \ref{thm:zon}. For a two dimensional polytope $P\sub \R^2$, our choice of $\eta$ was only limited to the condition that  $F_\eta(P)=\{u_1\},\,F_{-\eta}(P)=\{u_2\}$ for some $u_1,u_2\in V(P)$, which is  true for almost any $\eta \in S^1$.
		More generally, for any polytope $P\sub \R^n$ for which
		\[\partial P \sub \frac{1}{2}(P+V(P)),\]
		our choice of $\eta$ in the proof requires only that $F_\eta(P)=\{u_1\},\,F_{-\eta}(P)=\{u_2\}$ for some $u_1,u_2\in V(P)$, which is true for almost any $\eta \in S^{n-1}$, and the choice of the zonotope $Z$ is only limited by $P \sub [u_1,u_2]+Z$. 
\end{remark}
	
Note that if one could choose, in Theorem \ref{thm:zon}, a very small $Z$, namely if for every $\epsl>0$ 
one could find such a $Z$ contained in $\epsl B_2^n$, then this would imply denseness of \VG~polytopes, with respect to the Hausdorff distance, within the class of all polytopes (and hence within the class of all convex bodies). However, this cannot hold true. Indeed, denote the Minkowski subtraction of two sets $A,B\sub\RR^n$ by
$$
A\ominus B=\{x\in\RR^n\,:\, x-B\sub \intr(A)\}.
$$
Our next proposition implies  that if $P\not\in {\rm VG}(\R^n,\lambda)$ then its Minkowski sum with any body in a small enough ball cannot be in ${\rm VG}(\R^n,\lambda)$.

\begin{proposition}\label{prop:lower bound}
Let $n\ge 1$ and $\lambda \in [1/(n+1), 1/2]$.	Let $P,Q\subseteq \R^n$ be polytopes. Then
	\[\left(P\setminus\left((1-\lambda)P + \lambda V(P)\right)\right)\ominus Q \subseteq \left(P+Q\right) \setminus\left((1-\lambda)(P+Q) + \lambda V(P+Q)\right).\]
	In particular, if $P+Q \in {\rm VG}(\R^n,\lambda )$ then no translate of $Q$ can  fit into $P\setminus\left((1-\lambda)P + \lambda V(P)\right)$. 
\end{proposition}
\begin{proof}
	It is clear that $\left(P\setminus\left((1-\lambda)P + \lambda V(P)\right)\right)\ominus Q\subseteq P \ominus Q \sub P+Q$, so we need only show that this set includes no points in $ \left((1-\lambda)(P+Q) + \lambda V(P+Q)\right)$.
	The latter is included in $Q+(1-\lambda)P+\lambda V(P)$, so in particular for any point $x\in  \left((1-\lambda)(P+Q) + \lambda V(P+Q)\right)$  there exists some $y\in Q$ such that $x-y\in (1-\lambda)P+\lambda V(P)$, namely $x-Q \not\subseteq P\setminus ((1-\lambda)P+\lambda V(P))$. This means  that $x\not\in (P\setminus\left(  (1-\lambda)P+\lambda V(P)\right))\ominus Q$, as claimed. 
\end{proof}

\section{Vertex generated symmetric polytopes}\label{sec:anpotherproofforzn}\label{sec:central_sym}

Restricting to the class of centrally-symmetric polytopes enables us to prove additional properties regarding $\lambda$-\VG~polytopes.  

We  first show that for a centrally symmetric polytope, being $\lambda$-\VG~is equivalent to all of its facets being $\lambda$-\VG, namely that for centrally symmetric $P$, the converse of Lemma \ref{lem:fac} holds. This gives an easy inductive proof for the fact that zonotopes are \VG, reproving Corollary \ref{cor:zonotope}. After establishing this fact, we show that any $n$-dimensional $\lambda$-\VG~polytope can be realized as a facet of a centrally symmetric $(n+1)$-dimensional $\lambda$-\VG~polytope. In particular this means that if one knows  that centrally symmetric $\lambda$-\VG~polytopes in dimension $n+1$ are closed under Minkowski addition, then so is the class of all $\lambda$-\VG~polytopes in $\RR^n$. While this fact is yet to be proven, we do show that if two centrally symmetric $\lambda$-\VG~polytopes are in generic position then their Minkowski sum is also $\lambda$-\VG~(we explain the notion of a generic pair of polytopes in detail below).  Finally we prove a statement regarding the sum of the vertices $V(P)$ of a polytope with some linear image of $P$, a special case of which is the following curious fact: If $P-V(P)$ is a convex set, then $P$ is centrally symmetric and \VG, namely $P-V(P)=P+V(P)=2P$. 

We start by showing that for centrally symmetric polytopes, the converse of Lemma \ref{lem:fac} holds. 

\begin{lemma}\label{lemma:csVG}
Let $n\ge 1$ and $\lambda \in [1/(n+1), 1/2]$.	Let   $P\sub \RR^n$ be a centrally symmetric polytope and assume that 
	$\partial P\subseteq (1-\lambda)P+\lambda V(P)$. Then $P$ is $\lambda$-vertex generated. 
\end{lemma}

\begin{proof}
	Given $x\in P$ let $c\ge 1$ such that $cx\in \partial P$.
	By our assumption there exists $v\in V(P)$  such that $cx\in \lambda v+(1-\lambda)P$. 
	However, since $P$ is centrally symmetric, also $0\in \lambda v+(1-\lambda)P$ (since $0 = \lambda v+ (1-\lambda)(-\frac{\lambda}{1-\lambda}v)$ and $-\frac{\lambda}{1-\lambda}v\in P$). The set $\lambda v+(1-\lambda)P$ is convex, and so together with $0$ and $cx$ it includes $x = (1-1/c)0+ (1/c)cx$, and the proof is complete. 
\end{proof}

\begin{remark}\label{rem:weaker_than_Cen_Sym}
    In fact we see that the ``central symmetry'' condition is an overshoot for $\lambda<1/2$, and we can ask for less, $-\lambda P \sub (1-\lambda)P$.
    Note that we always have this for $\lambda = 1/(n+1)$ since  $-P\sub nP$. 
    This observation implies, however, that if for example,  $-P\sub (n-1)P$, then (using that the facets are $1/n$-\VG) the whole polytope is $1/n$-\VG. 
\end{remark}

Using Lemma \ref{lemma:csVG} we can provide a simple proof for Corollary \ref{cor:zonotope}. 

\begin{proof}[Another proof of Corollary \ref{cor:zonotope}]
Recall that a zonotope is centrally symmetric, and all of its faces are centrally symmetric zonotopes of lower dimension (in fact, a polytope is a zonotope if and only of all of its two-dimensional faces are centrally symmetric, see \cite{Sch_book}*{Theorem 3.5.2}). We prove the corollary by induction, where clearly one dimensional zonotopes are \VG~since these are simply segments. If we know that $(n-1)$-dimensional zonotopes are \VG, and we are given an $n$-dimensional zonotope $Z$ then it is centrally symmetric, and its boundary $\partial Z$ is the union of finitely many zonotopes which are \VG~by induction, so that $\partial Z \subseteq (\partial Z+V(\partial Z))/2 \subseteq (  Z+V(Z))/2$. Applying Lemma \ref{lemma:csVG}, $Z$ is \VG~as well. 
\end{proof}

\begin{remark}
It is well-known  that all centrally symmetric polytopes in the plane $\RR^2$ are zonotopes, see e.g., \cite{Sch_book}*{Corollary 3.5.7} and hence \VG. However, in higher dimensions there exist centrally symmetric polytopes which are not zonotopes. One such example is the cross polytope 
for $n \geq 3$, which has faces which are simplices and hence is  not  \VG. On the other hand, there also exist centrally symmetric \VG~polytopes which are not zonotopes, and one such example is the sum of the cross polytope with a suitable chosen zonotope. Indeed, one may choose a suitable zonotope by Theorem \ref{thm:zon},  and the fact that the cross polytope summed with a zonotope is not a zonotope follows e.g.~from \cite{Sch_book}*{Corollary 3.5.7}. 
\end{remark}	
	
Next we show that every \VG~polytope in $\RR^n$ can be realized as a facet  of a centrally symmetric \VG~polytope in dimension $(n+1)$. 

\begin{proposition}
	Let $n\ge 1$ and $\lambda \in [1/(n+1), 1/2]$, and let $P\in {\rm VG}(\R^n, \lambda)$. Then  there exists a 
 centrally-symmetric  $Q\in {\rm VG}(\R^{n+1}, \lambda)$ 
	such that $P = F_u(Q)$ for some $u\in S^{n}$. (In fact, all the $(n-1)$-dimensional faces of $Q$ which are not in directions $\pm u$ will be in ${\rm VG}(\R^n)$). 
\end{proposition}
\begin{proof}
	Given $P\in {\rm VG}(\R^n)$, define $P'=\conv(P\times\{1\},-P\times\{-1\})\subseteq \R^{n+1} = \RR^n \times \RR$. 
	The polytope $P'$ is centrally symmetric and its facets in directions $e_{n+1}$ and $-e_{n+1}$ are translates of $P$ and $-P$ respectively (where $(e_i)_{i=1}^{n+1}$ is the standard vector basis in $\RR^n \times \RR$.) Moreover, for any $(u_j)_{j=1}^m\sub  S^{n}\setminus e_n^\perp$ and $(c_j)_{j=1}^m \sub \RR$, the polytope $Q = P'+\sum_{j=1}^m c_j[-u_j, u_j]$ satisfies that its facets in direction directions $e_{n+1}$ and $-e_{n+1}$ are translates of $P$ and $-P$ respectively (again, by using \eqref{eq:Fu_sum}, say). It is also important to note that all faces of $P'$ which are {\em not} in the hyperplanes $H_1 = \{x_{n+1} = 1\}$ $H_{-1} = \{x_{n+1} = -1\}$, are not parallel to these hyperplanes (equivalently, are not orthogonal to $e_{n+1}$). Indeed, a face of $P'$ is a convex hull of some subset of its vertices, and these vertices belong to $H_{1} \cup H_{-1}$, so if the subset includes at least one element from each of the hyperplanes, then the corresponding face includes the edge between these two vertices, which is not orthogonal to $e_{n+1}$.  
	
	Our task is to   choose $(u_j), (c_j)$ such that $Q$ is $\lambda$-\VG. To this end, we will make sure that all of its faces are $\lambda$-\VG~and use Lemma \ref{lemma:csVG}. We will go over all the faces of 
 $P'$ which are not included in the hyperplanes $H_1$ and $H_2$ (as these are automatically $\lambda$-\VG) and use, for each one, the construction in the proof of Theorem \ref{thm:zon}. In other words, for each such face  $F$ of $P'$ we find a zonotope $Z_F\sub {\rm affine}(F)$ such that $F+Z_F \in {\rm VG}(\RR^{n+1},\lambda)$. (If $F$ was 
	$\lambda$-\VG, we pick $Z_F = \{ 0\}$). Moreover, we can, using Remark \ref{rem:dim}, choose  all  the vectors $\xi_i$ participating in the construction of $Z_F = \sum_{i=1}^{m_F} [-\xi_i, \xi_i]$ so that they are not in $e_{n+1}^\perp$ (here it is essential that $F$ is not orthogonal to $e_{n+1}$.  
	We let $Q = P' + \sum_{F} Z_F$ where $F$ runs over all of the  faces of $P'$ which are not orthogonal to $e_{n+1}$. 
	We claim that all faces of $Q$ are $\lambda$-\VG. Indeed, let $u\in S^n$, and denote $F^*$ the face of $P'$ in direction $u$. Employing \eqref{eq:Fu_sum}, as usual, we have 
	\[ F_u (Q) = F^* + \sum_F F_u (Z_F)  = F^* + F_u (Z_{F^*}) + \sum_{F\neq F^*} F_u(Z_F)
	\]
	The set $F^* + F_u (Z_{F^*}) = F^* +  Z_{F^*}$ is $\lambda$-vertex generated and it is summed with zonotopes, so by Theorem \ref{thm:sum} the face $F_u(Q)$ is $\lambda$-\VG~as well. This completes the proof.  
	\end{proof}

We proceed with proving Proposition \ref{prop:sum_of_generic_VG}, namely that the sum of a generic pair (see Definition \ref{def:gen-pair}) of centrally symmetric $\lambda$-\VG~polytopes is  $\lambda$-\VG~as well.

We remark that we can relax the condition that the polytopes are centrally symmetric and assume instead that the polytopes satisfy the conditions asserted in  Lemma \ref{lemma:csVG}, see Remark \ref{rem:weaker_than_Cen_Sym}.

\begin{proof}[Proof of Proposition \ref{prop:sum_of_generic_VG}]
Since $P+Q$ is centrally symmetric, by Lemma \ref{lemma:csVG} it suffices to show that its facets are $\lambda$-\VG, which is equivalent to $\partial(P+Q)\subseteq (1-\lambda)(P+Q)+ \lambda V(P+Q)$.   
  Given $x\in \partial (P+Q)$, there exists a unit vector $u\in S^{n-1}$ such that 
		$x\in F_u(P+Q) = F_u(P) + F_u(Q)$ (by  \eqref{eq:Fu_sum}, as usual).
	%
		 Let $x=x_1+x_2$ where $x_1\in F_u(P)$ and $x_2\in F_u(Q)$. As $P,Q$ are $\lambda$-vertex generated, so are their faces, (see Lemma \ref{lem:fac}) namely there exist $v_1\in V(F_u(P))$, $y_1\in F_u(P)$, $v_2\in V(F_u(Q)$ and $y_2\in F_u(Q)$ 
		 such that  $x_1=(1-\lambda) y_1 + \lambda v_1$ and 		 
		  $x_2=(1-\lambda) y_2 + \lambda v_2$. 
		  Clearly 
		    $u\in N_P(v_1)\cap N_Q(v_2)$, meaning in particular that $N_P(v_1)\cap N_Q(v_1)\ne \emptyset$. Since $P,Q$ are assumed to be a generic pair, this implies that $\intr(N_P(v_1))\cap \intr(N_Q(v_1))\ne \emptyset$. In such a case it is easy to check (see e.g.~\cite{Brvnk_book}*{Chapter 6, Lemma 1.3}) that  $v_1+v_2\in V(P+Q)$ (as they are the Minkowski sum of the faces of $P$ and $Q$ in direction $w\in \intr(N_P(v_1))\cap \intr(N_Q(v_1))$, say). We see thus that 
		\[ x=(1-\lambda) (y_1 + y_2) + \lambda (v_1+v_2) \in (1-\lambda) (P+Q)+\lambda V(P+Q).\] 
	This shows that $\partial(P+Q)\subseteq (1-\lambda) (P+Q)+\lambda V(P+Q)$, and by Lemma \ref{lemma:csVG}
	  $P+Q$ is $\lambda$-vertex generated.
	\end{proof}

 In the remainder of this section we focus on $\lambda = 1/2$. 
The idea behind our next theorem is based on the following attempt to generalize the notion of \VG~polytopes. What would happen if we asked for a polytope to satisfy, instead of $P+ V(P) = 2P$, the relation 
\[ P- V(P) = P- P.\]
 Clearly there is an inclusion of the left hand side in the right hand side, and in the special case of centrally symmetric polytopes, this is yet again the definition of a \VG~polytope. It turns out that in fact there is no other instance where this equality can hold. In fact, much more can be said. The mere requirement that $P-V(P)$ is a convex set, already implies that $P$ is centrally symmetric and \VG. Moreover, the operation $K\mapsto -K$ can be replaced in this claim by any $K\mapsto AK$ for any $A\in GL_n(\RR)$ such that $A^k=Id$ for some positive integer $k$.

\begin{theorem}\label{thm:PplusAP}
	Let $P\sub \R^n$ be a polytope and let $A\in GL_n(\RR)$ satisfy $A^k=Id$ for some natural number $k>0$. Then the following statements are equivalent:	\begin{enumerate}
		\item\label{1d} $P+V(AP)$ is convex
		\item\label{2d} $P$ is \VG~ and $P=AP+x$ for some $x\in\RR^n$.  
	\end{enumerate}
	In particular, if $P-P=P-V(P)$ then $P$ is a centrally symmetric vertex generated polytope. 
\end{theorem}

To prove Theorem \ref{thm:PplusAP}, we  need the following simple lemma.
\begin{lemma}\label{lem:face_sum_vg} Let
	$P,Q\sub\RR^n$ be two polytopes  and let $u\in S^{n-1}$. Suppose   that $P+V(Q)=P+Q$. Then  $$F_u(P+Q)=F_u(P)+V(F_u(Q))$$
	and, in particular, $\dim(F_u(P+Q))=\dim(F_u(P))$.
\end{lemma}
\begin{proof}
	First note that by our assumption and \eqref{eq:Fu_sum}, 
	\[F_u(P+V(Q))=F_u(P+Q)=F_u(P)+F_u(Q)\supseteq F_u(P)+V(F_u(Q)).\]
To prove the reverse inclusion, let $x\in F_u(P+V(Q))$. By definition, there exist $y\in P$ and $v\in V(Q)$ such that $x=y+v$ and  $\langle u,x\rangle=h_{P+Q}(u)$. By the additivity of the support function with respect to Minkowski addition (see e.g., \cite{Sch_book}*{Theorem 1.7.5}), we therefore have 
	\[h_P(u)+h_Q(u)=h_{P+Q}(u)=\langle u,x\rangle =\langle u,y\rangle+\langle u,v\rangle.\]
Since  $\langle u,y\rangle\le h_P(u)$ and $\langle u,v\rangle\le h_Q(u)$, it follows that  $\langle u,y\rangle=h_P(u)$ and $\langle u,v\rangle= h_Q(u)$.  We thus conclude that, $y\in F_u(P)$ and $v\in V(F_u(Q))$, which completes our proof. 
\end{proof}

\begin{proof}[Proof of Theorem \ref{thm:PplusAP}]
	First note that \eqref{2d} trivially implies \eqref{1d}. We proceed to prove that \eqref{1d} implies \eqref{2d}. 
	Since $P+V(AP)$ includes all the extremal points of  $P + AP$, it follows that  $\conv(P+V(AP))=P+AP$, and so $P+V(AP)$ is convex implies that $P+V(AP)=P+AP$.  
	
We next show  that  for any $ u\in S^{n-1}$, $\dim (F_{Au}(P))=\dim (F_u(P))$. Indeed, using that $P+AP=P+V(AP)$, that  $F_{Au}(AP)=AF_u(P)$ and  Lemma \ref{lem:face_sum_vg}, we have 
\begin{multline}\label{dimansion}
\dim (AF_{u}(P))\le \dim(F_{Au}(P)+AF_{u}(P))=\dim (F_{Au}(P+AP))=	\dim (F_{Au}(P)).
\end{multline}
Therefore, $\dim(F_u(P))\le\dim(F_{Au}(P))$.
Since  \eqref{dimansion} holds for any direction $u$, using our assumption that $A^k = I$,  we  see that
	\begin{equation}\label{recursion}
		\dim (F_u(P))\le\dim (F_{Au}(P))\le\dots\le\dim(F_{A^ku}(P))=\dim(F_{u}(P)),
	\end{equation}
and so $\dim (F_{u}(P))=\dim (F_{Au}(P))$, as claimed, and the inequality in \eqref{dimansion} is in fact an equality.

Note that we can also infer 
that $\aff(F_{Au}(P))$ is a translate of $\aff(AF_u(P))$ since the equality $\dim(F_{Au}(P)+AF_{u}(P))= \dim (AF_{u}(P))$  implies that $\aff(F_{Au}(P))$ is a subset of a translate of $\aff (AF_{u}(P))$ and hence the equality  $\dim (AF_{u}(P))=\dim (F_{Au}(P))$  implies that their affine hulls coincide.
	
	Next we show that if $\dim(F_{u}(P))=1$ then $\vol_1(F_{Au}(P)) = \vol_1(F_u(P))$. Indeed, 
	 if $\dim(F_{u}(P))=1$     
	 then,  by \eqref{dimansion},  $\dim(F_{Au}(P))=\dim(F_{Au}(P+AP))=1$.
	In particular, since $F_{Au}(P+AP)=F_{Au}(P)+AF_u(P)$, it follows that
	  $$\volk{1} (F_{Au}(P+AP))=\volk{1}(F_{Au}(P))+\volk{1}(AF_u(P)).$$
Moreover, by Lemma \ref{lem:face_sum_vg}, we have (since $F_{Au}(AP)$ is one-dimensional)
$$
\volk{1}(F_{Au}(P+AP))=\volk{1}(F_{Au}(P)+V(F_{Au}(AP)))\le 2\volk{1}(F_{Au}(P))
$$
and hence $ \volk{1}(F_u(P))\le \volk{1}(F_{Au}(P))$ for any $u$.  As in \eqref{recursion}, using the fact that 
$A^k=Id$, we obtain that  $\vol_1(F_{Au}(P)) = \vol_1(F_u(P))$.
	
So far, we have established  that for any $u\in S^{n-1}$ such that $\dim(F_u(P))=1$,  $AF_u(P)$ and $F_{Au}(P)$ are translates of one another. Next, we prove the following claim:
\begin{equation}\label{eq:Anormal_cone}
\forall F\in{\mathcal F}(P)\,\,\exists G\in{\mathcal F}(P)\text{ such that } An_P(F)=n_P(G).
\end{equation}
 Indeed, first note that for every $u\in\relint(n_P(F))$, we have $F=F_u(P)$ and that, by definition, the finite set
$
\{F_{Au}\,:\,u\in\relint(n_P(F))\}:=\{F_1,\dots,F_k\}
$
satisfies  that $$A\relint(n_P(F))\sub\bigcup_{i=1}^k\relint(n_P(F_i)).$$
As shown in \eqref{recursion} we know that $\dim(F)=\dim(F_i)$ for all $i$, which means ${\rm span}(An_P(F))={\rm span}(n_P(F_i))$.  However, in this spanned subspace (which is simply $(AF)^\perp$) the sets  given by
$\relint(n_P(F_1)),\dots, \relint(n_P(F_k))$ are pairwise disjoint and open  (relative to the subspace). Since $A\relint(n_P(F))$ is a connected set, it cannot be covered by disjoint open sets. Therefore  $F_1=\dots =F_k$. In other words,  $A\relint(n_P(F))\sub\relint(F_{Au})$ for every  $u\in\relint(n_P(F))$. Applying  $A$ and using this inclusion repeatedly, we obtain 
$$A^{k+1}\relint(n_P(F))\sub A^{k}n_P(F_{Au}(P))\sub A\relint(n_P(F_{A^ku}(P))).$$
Since $A^k=I$ and $F=F_u(P)$, it  follows that $A\relint n_P(F)=\relint n_P(F_{Au}(P))$ and hence (as normal cones are closed) $An_P(F)=n_P(F_{Au}(P))$, as claimed. 

Next, we claim  that
\begin{equation}\label{eq:FsubAsub}
F_w(P)\sub F_u(P)\implies F_{Aw}(P)\sub F_{Au}(P).
\end{equation} 
Indeed, suppose that $u\in n_P(F_u(P))\sub n_P(F_w(P))$. By applying $A$ on both sides, we get 
$$Au\in An_P(F_u(P))\sub An_P(F_w(P)).$$  
By \eqref{eq:Anormal_cone},   $An_P(F_u(P))=n_P(G)$ and $An_P(F_w(P))=n_P(G')$ for some faces $G,G'$ of $P$ with $G'\sub G$. Therefore , it follows that 
$n_P(F_{Au}(P))\sub n_P(G)$ and $n_P(F_{Aw}(P))\sub n_P(G')$. Since, by \eqref{recursion}, $\dim G=\dim F_u(P)=\dim F_{Au}$, we have $F_{Au}(P)=G$, and similarly $F_{Aw}(P)=G'$.  Thus, we obtain that $F_{Aw}(P)=G'\sub G= F_{Au}(P)$, as claimed. 

Let $F_u$ be a $1$-dimensional face of $P$ and let $F_v$ one of its vertices. By \eqref{recursion} and \eqref{eq:FsubAsub}, $F_{Au}$ is also a $1$-dimensional face of $P$ with $F_{Av}$ as one of its vertices.  Moreover, as we already established, $F_{Au}=AF_u+x_u$ for some $x_u\in\RR^n$, and (trivially) $F_{Av} =AF_v+x_v$ for some $x_v\in\RR^n$. Our goal is to show that $x_u=x_v$, from which it readily follows that $V(AP)=V(P)+x$ for some $x\in\RR^n$ (as all vertices are connected via $1$-dimensional faces). Indeed, denote $E=F_u(P)$ and $V=F_v(P)$.  Note that $AE+x_u$ and $AE+x_v$ are parallel line segments of the same length and with a common vertex {$AV+x_v$} (as $F_{Av}$ is a vertex of $F_{Au}$). Therefore,  these segments are either  identical, namely $x_u=x_v$ or consecutive so that their union $T$ is a segment and their intersection is {$AV+x_v$}. Assume the latter.  In particular, we have $AV+x_v\in\relint(T)$.  Denote $H=(Av)^\perp$ and $H^-=\{x\in\RR^n\,:\,\iprod{x}{Av}\le0\}$. Since $Av$ is a normal of $P$ at the vertex $AV+x_v$ and $AE+x_u$ is a face of $P$, we have $AE+x_u\sub H^-+AV+x_v$. On the other hand, clearly $F_{Av}(AP+x_v)=AV+x_v$ and $AE+x_v$ is a face of $AP+x_v$, which means that $AE+x_v\sub H^-+AV+x_v$ and so $T\sub H^-+AV+x_v$. However, since $AV+x_v\in\relint(T)$, it follows that $Av$ must be orthogonal to $T$, which contradicts the fact that $\dim(F_{Av})=0$. Thus, we have $x_u=x_v$.

Concluding the above, we have $V(AP)=V(P)+x$ for some $x\in\RR^n$, and so clearly $AP=P+x$.
Since  $P+V(AP)$ is assumed to be convex and $V(P)=V(AP)$,  $P$ must also be vertex generated, which completes our proof.
\end{proof}

\section{A series expansion and covering estimates}\label{series}
In this section, we discuss two more  properties of $\lambda$-\VG~polytopes which are straightforward from the definition, and are the reason for our choice of name for this class.

We begin with the property, explained in the introduction, characterizing members of the class  $ {\rm VG}(\RR^n)$ as polytopes $P$ that can be written as the closure of a certain infinite sum involving the vertices of $P$. 

\begin{proposition}\label{prop:infinite-sum}
 Let $n\ge 1$ and $\lambda \in [1/(n+1), 1/2]$.  Let $P\sub \RR^n$ be a polytope. Then $P\in {\rm VG}(\RR^n, \lambda)$ if and only if 
 \[ 
 P = \cls\left({\sum_{i=0}^\infty (1-\lambda)^{i} \lambda V(P)}\right). \]   
\end{proposition}

\begin{proof}
For any polytope, using that $P$ is closed and convex and that 
${\sum_{i=0}^\infty (1-\lambda)^{i} \lambda } = 1$ we have
	\begin{equation}\label{eq:simple-convexity}
		{\sum_{i=0}^\infty (1-\lambda)^{i} \lambda V(P)} \subseteq 
  P
	\end{equation}
	so that the inclusion  $P \supseteq  \cls\left({\sum_{i=0}^\infty (1-\lambda)^{i} \lambda V(P)}\right)$ holds without any assumptions. We thus need to show that the opposite inclusion holds if and only if $P\in {\rm VG}(\RR^n, \lambda)$.
	
	 Assume $P\in {\rm VG}(\RR^n, \lambda)$. Let $R>0$ be such that $P\sub RB_2^n$.  
 Using that $P\in {\rm VG}(\RR^n, \lambda)$ we see inductively that for every $k\in {\mathbb N}$ it satisfies 
 $P = \sum_{i=0}^k (1-\lambda)^{i} \lambda V(P)+ (1-\lambda)^{k+1}P $.
 Given $x \in P$ and $\epsl>0$, we choose $k\in {\mathbb N}$ such that $R(1-\lambda)^{k+1}< {\epsl}/{2}$. Then 
\begin{equation}
    P = \sum_{i=0}^k (1-\lambda)^{i} \lambda V(P)+ (1-\lambda)^{k+1}P \subseteq \sum_{i=0}^k (1-\lambda)^{i} \lambda V(P) +\frac \epsl2 B_2^n.
\end{equation}
Therefore there exists $v_1 \in \sum_{i=0}^k (1-\lambda)^{i} \lambda V(P)$ such that $|v_1-x|\le \epsl/2$.
From \eqref{eq:simple-convexity} we know that 
\begin{equation}
    \sum_{i=k+1}^\infty (1-\lambda)^{i} \lambda V(P)=(1-\lambda)^{k+1}{\sum_{i=0}^\infty  (1-\lambda)^{i} \lambda V(P)}\subseteq (1-\lambda)^{k+1}P\subseteq \frac \epsl2 B_2^n.
\end{equation}
Therefore for any $v_2 \in \sum_{i=k+1}^\infty (1-\lambda)^{i} \lambda V(P)$ we have that $|v_2|<\frac \epsl2$.
Thus we may conclude that 
\begin{equation}
    d(x,\sum_{i=0}^\infty  (1-\lambda)^{i}\lambda V(P))\le |(v_1+v_2)-x|\le |v_1-x|+|v_2|<\epsl,
\end{equation}
proving the inclusion  $P \subseteq  \cls\left({\sum_{i=0}^\infty (1-\lambda)^{i} \lambda V(P)}\right)$. 

For the other direction, assume $P = \cls\left({\sum_{i=0}^\infty (1-\lambda)^{i} \lambda V(P)}\right)$. Write
\begin{eqnarray*}
    P &=& \cls\left({\sum_{i=0}^\infty (1-\lambda)^{i} \lambda V(P)}\right) 
   = \cls \left(  \lambda  V(P)+ (1-\lambda) \sum_{i=0}^\infty (1-\lambda)^{i} \lambda V(P)
      \right)
     \\
   &  \subseteq&\cls\left(\lambda V(P)+(1-\lambda)  P\right)=\lambda V(P)+(1-\lambda)  P, 
\end{eqnarray*}
where the last equality is a result of $\lambda V(P)+(1-\lambda)  P$ being a finite union of closed sets. Since the opposite inclusion 
$\lambda V(P)+(1-\lambda)  P\subseteq P$, holds trivially (for any polytope) it follows that $P\in {\rm VG}(\RR^n, \lambda)$. 
\end{proof}
\begin{remark}
    It is instructive to note that when $P$ is the standard simplex in $\RR^2$, the set $\cls\left({\sum_{i=1}^\infty 2^{-i} V(P)}\right)=\bigcap_k(\sum_{i=1}^k 2^{-i}V(P)+ 2^{-k}P)$ is the Sierpinski triangle. For other $P\not\in {\rm VG}(\RR^n)$ one may obtain other fractal-like objects.
\end{remark}

Another useful fact about the class of $\lambda$-\VG~polytopes is that it has good covering properties. This should not come as a surprise, as the definition of the class itself is that $|V(P)|$ copies of $(1-\lambda)P$ form a cover for $P$. 
Recall the notion of covering numbers.

\begin{definition}[Covering numbers]
	Let $K$ and $T$ be convex bodies in $\R^n$, The Covering number $N(K,T)$ of $K$ by $T$ is defined as follows:
	\[N(K,T):=\min\left\{N \in \mathbb{N} \mid\exists x_1,\dots,x_N \in \R^n : K \subseteq \bigcup_{i=1}^N x_i+T \right\}.\]
\end{definition}
The results of this paper imply a simple covering estimate for the rich class of vertex generated polytopes, which is in the spirit of Maurey's lemma.
\begin{lemma}[Maurey's lemma, \cite{Pisier81}*{Lemma 2}]
	Let $X$ be a space with type $p$, and unit ball $T$. Let $m$ be an integer and assume that $P$ is the convex hull of $m$ points in $T$. Then for any integer $k$
	\[N(P,2k^{-1/q}T_p(X)T)\le m^k\]
	where $q$ is the conjugate of $p$, that is $1/p+1/q=1$.
\end{lemma}
The proof of Maurey's lemma uses averages of the vertices of the polytope. In the case of \VG-polytopes, we can similarly provide a net using weighted averages (indeed, this is almost the definition of \VG). 

Recall that $\lambda(P) = \sup \{\lambda\,:\, P=(1-\lambda) P+\lambda V(P)\}$, and denote $a(P)=1-\lambda(P)$.
\begin{proposition}
If  $P\sub \R^n$ is a polytope with $m$ vertices then for any $k\in { \mathbb N}$,
		\[N\left(P,a(P)^kP\right)\le m^k.\]
	\end{proposition}
	\begin{proof}
Set $a:=a(P)$.  In the same spirit of the proof of Proposition \ref{prop:infinite-sum}, we note that  for any $k\in {\mathbb N}$, we have
		\[P=a  P+ (1-a) V(P)=a^kP+(1-a)\sum_{i=0}^{k-1}a^iV(P).\]
As the number of points in $\sum_{i=0}^{k-1}a^iV(P)$ is at most $m^k$, we get the desired result.
	\end{proof}
\noindent Note that on the one hand, by \eqref{eq:carat}, it holds for any polytope $P\sub\RR^n$ that 
\[N\left(P,\left(\frac{n}{n+1}\right)^kP\right)\le m^k.\]
On the other hand, for a vertex generated polytope $P\in {\rm VG}(\R^n)$ (in particular, for any zonotope $P$), we get the superior covering estimates $N(P, 2^{-k} P)\leq m^k$.

Also note that we have the volume lower bound 
	\[ 
		N(P, 2^{-k} P)\geq \frac {\vol (P)}{\vol (2^{-k}P)} = 2^{nk} = (2^n)^k.
	\]
	Since, by Lemma \ref{lem:vert}, for a \VG~polytope we have $m\ge 2^n$, with equality precisely when $P$ is a parallelopiped, the proposition  can be interpreted as the fact that when the number of vertices of the \VG~polytope is not much larger that $2^n$, the volume lower bound is close to being an equality.

\section{Some concluding remarks}\label{sec:other_stuff}

Let us describe yet another angle from which to approach \VG-polytopes.  Motivated by studying Brunn-Minkowski type inequalities for sums of boundaries \cite{AFS-BMtype}, we show in \cite{AFS-cara} (see also, \cite{Tomer-thesis}) the following theorem for $n$-dimensional polytopes. 

\begin{theorem}\label{thm:bd-and-k} Let $n\ge 1$. 
		For any polytope $P\sub\R^n$ it holds that
		\[P=\frac{\partial^{\lceil\frac n2\rceil}P+\partial ^ {\lfloor\frac n2\rfloor}P}{2}.\]
	\end{theorem}
Here for a polytope $P \sub\RR^n$, we denoted the union of its $k$-dimensional faces by $\partial^k P$.  
Moreover, we showed that if $2P=\partial^{k}P+\partial^{n-k}P$ for some $k\in \{0, \ldots, \lfloor n/2 \rfloor \}$ then  $2P=\partial^{m}P+\partial^{n-m}P$  for all $k\le m\le  \lfloor n/2 \rfloor$. 

This allows to   define, for every polytope, its ``critical dimension'' $k^*(P)$ which is the smallest $k$ for which 	$2P=\partial^{k}P+\partial^{n-k}P$.  Theorem \ref{thm:bd-and-k} guarantees that $k^*$ exists and is at most $\lfloor{n/2}\rfloor$, see again \cite{AFS-cara}. 
A simplex  $\Delta_n\subset \RR^n$ is an example of a  polytope with largest possible critical dimension, $k^*(\Delta) = \lfloor{n/2}\rfloor$. The class of \VG-polytopes is precisely the class of polytopes for which the critical dimension is $k^* = 0$.

\begin{remark} It is worth mentioning that the $\lambda$-parameter we introduced to measure closeness of a polytope to being \VG~does not have a similar straightforward analogue for the class of polytopes for which  
	  $P=\frac 12(\partial^kP+\partial^{n-k}P)$,  $0<k<n$. Indeed, asking for an identity of the form $P=\lambda\partial^kP+(1-\lambda)\partial^{n-k}P$ to hold, for some $1/2\neq \lambda$, already fails for say $P=[-1,1]^n$ the unit cube. In other words, for the cube $P$, we have that $P=\lambda\partial^kP+(1-\lambda)\partial^{n-k}P$ for a given $0<k<n$ if and only if $\lambda=\frac 12$.
	Indeed, assume without loss of generality that $\lambda>\frac 12$, and let us show that $0\notin \lambda\partial^kP+(1-\lambda)\partial^{n-k}P$.
	Assume this statement is false, so there exist $x\in \partial^kP$ and $y\in \partial^{n-k}P$ such that $\lambda x+(1-\lambda)y=0$, because $x\in \partial P$ there exist $1\le i\le n$ such that $|\langle x,e_i\rangle|=1$, so without loss of generality     assume that $\langle x,e_1\rangle =1$, this means that
	\[ (1-\lambda)\langle y, e_1\rangle=\langle 0, e_1\rangle-\lambda\langle x, e_1\rangle\] and hence
	\[\langle y, e_1\rangle=0-\frac{\lambda}{1-\lambda}\langle x, e_1\rangle<-1,\] 
 which is a contradiction as $y\in [-1,1]^n$.
\end{remark}

When considering higher dimensional boundary parts instead of vertices, one can improve the factor $1/(n+1)$ from \eqref{eq:carat} significantly. The fact that $\partial P$ is connected, together with a classical result of Fenchel, it follows that 
 for every polytope $P\sub \R^n$  
 	\[P=\frac{n-1}{n} P+\frac 1{n}\partial^1P.\]
 Indeed, Fenchel \cite{Fenchel1929} showed that if a set $A$ cannot be separated into two disconnected parts by a hyperplane (which does not intersect $A$), then any point $x\in {\rm conv}(A)$ can be written as the convex hull of $n$ points from $A$. The rest of the proof is just as we showed \eqref{eq:carat}. 

 More generally, B\'ar\'any and Karasev showed \cite{BK12}*{Corollary 2.4},  
 \begin{proposition}\label{general kG}
 	Let $P\sub \R^n$ be a polytope. Then 
 	\[P=\frac{n-k}{n-k+1} P+\frac 1{n-k+1}\partial^kP.\]
 \end{proposition}

It is worth mentioning that the Shapley-Folkman Lemma is also a generalization of \eqref{eq:carat}, by plugging in $A_i = V(P)$ and $m=n+1$. 

\begin{theorem}[Shapley-Folkman Lemma]\label{thm:SF}
    Let $A_1, \ldots A_m\subset \RR^n$, $m\ge n$ and let 
    $x\in \sum_{i=1}^m {\rm conv} (A_i)$. Then there is some subset $ I = \{ i_1, \ldots , i_n\}\subseteq \{1, \ldots, m\}$ such that 
    \[   x\in \sum_{i\in I} {\rm conv} (A_i) + \sum_{i\not\in I} A_i.\]
\end{theorem}

We end this section with a  proof of the fact that the only polytope in $\RR^n$ for which $\lambda(P) = 1/(n+1)$ is the simplex. This was first proved by Schneider in \cite{Schneider-1678}.

\begin{proposition}\label{prop:simplex iff lambda is minimal possible}
	For a polytope $P\sub \RR^n$, 
if $\lambda(P) = {1}/{(n+1)}$ then $P$ is an $n$-dimensional simplex. 
\end{proposition}

\begin{proof} Assume $\lambda(P) = 1/(n+1)$. Using  \eqref{eq:carat} we see that 
  $P$ is not contained in any affine hyperplane of $\R^n$.
  For an $n$-dimensional simplex $S$ 
  with center of mass at the origin, the set $S\setminus (\lambda S+(1-\lambda)V(S))$ is relatively easy to analyze. 
Indeed, it is given by the intersection of $S$ with $(n+1)$ half-spaces. More precisely, assume that 
\[ S = \bigcap_{i=1}^{n+1} \{x: \iprod{x}{\theta_i} \le r_i \}\]	with $\theta_i \in S^{n-1}$.  Then, since the center of mass is assumed to be the origin, for the vertex $v_i$ opposite to the facet $\{ x\in S: \iprod{x}{\theta_i} = r_i\}$ we have $\iprod{v_i}{\theta_i} = -nr_i$. Therefore if $\lambda > 1/(n+1)$ we have 
\[ (1-\lambda)S + \lambda v_i = S \cap \{ x: \iprod{x}{\theta_i} \le r_i((1-\lambda) (1+n) - n)\}.\]
	Letting $r_i' = r_i(n-(1-\lambda)(1+n))\le 0$ we thus have (see Figure \ref{fig:Simplex}) 
	\[T:= S\setminus \cup_i ((1-\lambda)S + \lambda v_i )
\subseteq \cap 	\{ x: \iprod{x}{-\theta_i}\le r_i'\}  = - (n-(1-\lambda)(1+n))S.
	\]
In \input{simplex} particular, as $\lambda \to 1/(n+1)$, this set converges to the point $\{0\}$. It is easy to check (by translating $S$) that for a simplex with center of mass $m(S)$ we have, similarly, 
		\[ S\setminus \cup_i ((1-\lambda)S + \lambda v_i )
	\subseteq  - (n-(1-\lambda)(1+n))S +m(S)(1-n+(1-\lambda)(1+n))
	\]
	Again, as $\lambda\to 1/(n+1)$, this set converges to a point, namely $m(S)$.

	Given some polytope $P$ that is not a simplex, consider the set of all simplices with vertices that are a subset of $V(P)$. We pick $\lambda >1/(n+1)$ close enough to $1/(n+1)$ so that the sets  \[ T(S): = S\setminus ((1-\lambda)S + \lambda V(S))
	\subseteq  - (n-(1-\lambda)(1+n))S +m(S)(1-n+(1-\lambda)(1+n))\]
	do not intersect for any pair of simplices $S_1$ and $S_2$ which differ by only one vertex (and so, have in particular different centers of mass). This is clearly possible by the convergence of $T(S)$ to $m(S)$ which we have demonstrated. 
	
	Let $x\in P$. Then $x$ belongs to two simplices with vertices in $V(P)$ that differ by only one vertex. Indeed, $x$ belongs to some simplex $S$ (if it belongs to a lower dimensional simplex, then the assertion is trivial). Consider any vertex $v$ of $P$ which is not participating in $S$, and   the ray emanating from $v$ in direction $x$. This ray intersects $S$ at two points, one of which, $x'$, satisfies that $x\in [x',v]$. Since $x'\in \partial S$ it is in the convex hull of $(n-1)$ vertices of $S$, and along with $v$ these span another simplex $S'$ which includes $x$ and differs from $S$ by only one vertex. Therefore, $x$ cannot belong to both $T(S)$ and $T(S')$. If $x\not\in T(S)$, say, then 
	\[ x\in (1-\lambda)S + \lambda V(S)\sub   (1-\lambda)P +\lambda V(P),	
		\]
as claimed (and similarly for $S'$).
\end{proof}

\bibliographystyle{amsplain}
\bibliography{ref}

\end{document}

%% file: dense2D.tex

\begin{figure}[ht]
	\centering
\begin{tikzpicture}[scale=2]
\def \aradius {3}
\def \bradius {5}
\def \cradius {10}
\def \anorm {45}
\def \bnorm {115}

\tkzDefPoint(1,0){A}
\tkzLabelPoint[below right,](1.1,0.15){\small $v_i$}
\tkzDefPoint(-1,0){B}
\tkzLabelPoint[below left](-1.1,0.15){\small $v_{i+1}$}
\tkzDefPoint(1,-5){C}
\tkzDefPoint(-3,-2){D}
\tkzDefPoint(85:\aradius){U1}

\tkzInterCC[R](A,\aradius)(B,\aradius) \tkzGetPoints{M1}{N1}
\tkzInterCC[R](A,\bradius)(C,\bradius)\tkzGetPoints{M2}{N2}
\tkzInterCC[R](B,\cradius)(D,\cradius)\tkzGetPoints{M3}{N3}

\tkzDrawArc[dashed, color=black](M1,A)(B)
\tkzLabelArc[color=black, above, yshift=-0.05cm](M1,A,B){\small $A_i$}
\tkzDrawArc[dashed](N1,B)(A)
\tkzDrawArc[rotate,dashed, color=black](N2,A)(-10)
\tkzDrawArc[rotate,dashed, color=black](M3,B)(5)
\draw[color=gray] (B)--++(225:1);
\draw[color=gray] (A)--++(270:1);
\draw[color=gray](A)--node[below, color=black, yshift=0.1cm]{\small$E_i$} (B);

\draw[dotted] (A)--++(\anorm-90:1);
\draw[dotted] (A)--++(\anorm+90:1);
\draw[dotted] (B)--++(\bnorm-90:1);
\draw[dotted] (B)--++(\bnorm+90:1);
\draw[->] (A)--node[right, xshift=0.4cm]{$\theta_i$}++(\anorm:1);

\draw[->] (B)--node[above right]{\small $\theta_{i+1}$}++(\bnorm:1);

\coordinate (U) at ($(M1)+(U1)$);

\draw(A)--(U)--(B);
\tkzDrawPoints(A,B,U)
\end{tikzpicture}
\caption{The construction of $Q'$.}
\label{fig:dense2D}
\end{figure}
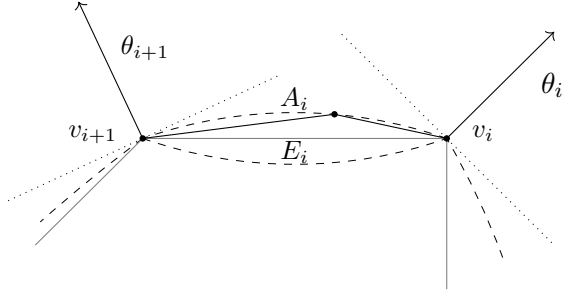

%% file: Simplex.tex
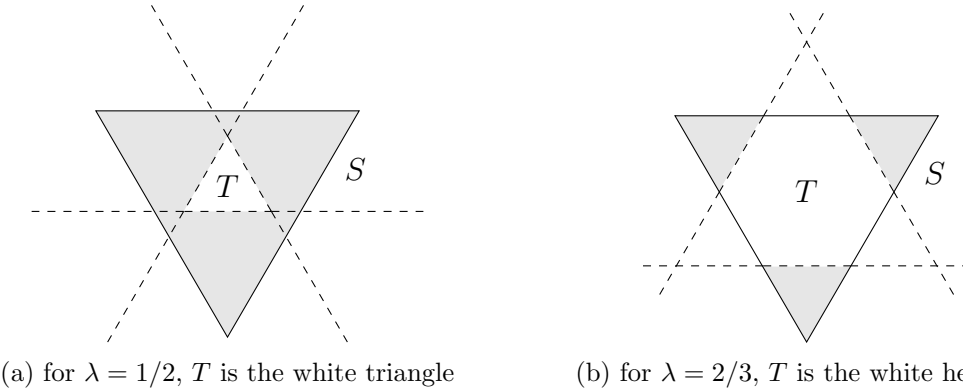
\begin{figure}
	\centering
	\begin{subfigure}{.5\textwidth}
		\centering
\begin{tikzpicture}[scale = 2]
	\fill[gray!20] (30:1) -- (270:1) -- (150:1) -- cycle;
	\fill[white] (330:1/3) -- (90:1/3) -- (210:1/3) -- cycle;
	\draw (30:1) -- (270:1) -- (150:1) -- cycle;
	\draw[dashed] (330:1/3) -- (90:1/3) -- (210:1/3) -- cycle;
	\draw[dashed] (330:1/3) --++(0:1);
	\draw[dashed] (330:1/3) --++(300:1);
	\draw[dashed] (90:1/3) --++(120:1);
	\draw[dashed](90:1/3) --++(60:1);
	\draw[dashed](210:1/3) --++(180:1);
	\draw[dashed](210:1/3) --++(240:1);
	\tkzLabelPoint[anchor=center](30:0){\large $T$}
\tkzLabelPoint[below right,](20:0.75){\large $S$}
\end{tikzpicture}
		\caption{for $\lambda=1/2$, $T$ is the white triangle}
		\label{fig:sub1}
	\end{subfigure}%
\begin{subfigure}{.5\textwidth}
	\centering
\begin{tikzpicture}[scale = 2]
	\fill[gray!20] (30:1) -- (270:1) -- (150:1) -- cycle;
	\fill[white] (330:1) -- (90:1) -- (210:1) -- cycle;
	\draw (30:1) -- (270:1) -- (150:1) -- cycle;
	\draw[dashed] (120:1/1.75) -- (180:1/1.75);
	\draw[dashed] (240:1/1.75) -- (300:1/1.75);
	\draw[dashed] (0:1/1.75) -- (60:1/1.75);
	\draw[dashed] (300:1/1.75) --++(0:0.8);
	\draw[dashed] (0:1/1.75) --++(300:0.8);
	\draw[dashed] (60:1/1.75) --++(120:0.8);
	\draw[dashed](120:1/1.75) --++(60:0.8);
	\draw[dashed](240:1/1.75) --++(180:0.8);
	\draw[dashed](180:1/1.75) --++(240:0.8);
	\tkzLabelPoint[anchor=center](30:0){\large $T$}
	\tkzLabelPoint[below right,](20:0.75){\large $S$}
\end{tikzpicture}
\caption{for $\lambda=2/3$, $T$ is the white hexagon}
\label{fig:sub2}
\end{subfigure}
 \caption{$T$ for different values of $\lambda$- always contained in some dashed simplex $-qS$.}.
\label{fig:Simplex}
\end{figure}